\documentclass[11pt,a4paper]{article}

\usepackage{amsmath,amsthm,amsbsy,bm,amssymb,enumerate}

\usepackage[latin1]{inputenc}
\usepackage[english]{babel}

\newcommand {\Z}	  {\mathbb{Z}}
\newcommand {\R}	  {\mathbb{R}}
\newcommand {\N}	  {\mathbb{N}}

\renewcommand{\epsilon}{\varepsilon}

\newcommand{\ve}{\boldsymbol}
\newcommand\supp{\mathrm{supp}}
\newcommand{\rank}{\mathrm{rank}}
\newcommand{\conv}{\mathrm{conv}}

\theoremstyle{plain}
\newtheorem{theorem}{Theorem}
\newtheorem{corollary}[theorem]{Corollary}

\title{The Support of Integer Optimal Solutions}
\author{I.~Aliev, J.~De~Loera, F.~Eisenbrand, T.~Oertel, and R.~Weismantel}

\date{\today }

\begin{document}
\maketitle
\begin{abstract}
  \noindent
%  We study the support of discrete
%  optimization problems  in standard form associated with an integral constraint matrix $A$.
 The support of a vector is the number of nonzero-components. We show that
  given an integral $m\times n$ matrix $A$, the integer linear optimization problem  $\max\left\{ \ve c^T\ve x : A\ve x = \ve b, \, \ve x\ge\ve 0, \,\ve x \in \Z^n\right\}$
  has an optimal solution whose support is bounded by 
  $2m \, \log (2 \sqrt{m} \| A \|_\infty)$, where $ \| A \|_\infty$ is the largest absolute value of an entry of $A$.  Compared to previous bounds, 
  the one presented here is independent on the objective function. We furthermore provide a nearly matching asymptotic lower bound on the support of optimal solutions. 
\end{abstract}

\section*{Introduction}

\noindent
We consider the  integer optimization  problem in standard form
\begin{equation}
  \label{eq:mainProblem}
  \max\left\{ \ve c^T\ve x : A\ve x = \ve b, \, \ve x\ge\ve 0, \,\ve x \in \Z^n\right\},
\end{equation}
where $A \in \Z^{m×n}$, $\ve b \in \Z^m $ and $\ve c \in \Z^n$.
For convenience, we will assume throughout this paper that Problem~\eqref{eq:mainProblem} is feasible.
We will also assume without loss of generality that the matrix $A$ has full row rank, i.e., $\rank(A)=m$. 

For a vector $\ve x \in \R^n$, let $\supp(\ve x) := \{ i \in \{1,\ldots,n\}: x_i \neq 0\}$ denote the \emph{support} of ${\ve x}$.
The main purpose of this paper is to establish lower and and upper bounds on the minimal size of  support of an optimal 
solution to Problem (\ref{eq:mainProblem}) which are polynomial in $m$ and the largest binary encoding length of an entry of $A$. 
Polynomial support bounds for integer programming \cite{EisenbrandShmonin2006,ADOO} have been successfully used in many areas such as in logic and  complexity, see \cite{kuncak2007towards,kieronski2014two} in the design of efficient polynomial-time approximation schemes \cite{jansen2010eptas,jansen2016closing}, in fixed parameter tractability \cite{knop2017combinatorial,onn2017huge} and they were an ingredient in the solution of cutting stock with a fixed number of item types \cite{goemans2014polynomiality}. These previous bounds however were tailored for the integer feasibility problem only and thus depend on the largest encoding length of a component of the objective function vector if applied to the optimization problem~\eqref{eq:mainProblem}.  
We believe that the observation that we lay out in this paper, namely that these bounds are independent of the objective function, will find  further applications in algorithms and complexity.

Note that the support of an optimal solution of the linear relaxation of \eqref{eq:mainProblem} is well-understood: if we denote by $\ve x^*$ 
a vertex of the polyhedron $P(A, {\ve b}):= \{\ve x \in \R: A\ve x = \ve b, \ \ve x \geq 0 \}$ that corresponds to the linear optimum, then
from the theory of linear programming we know that $\ve x^*$ is determined by a basis $B$ with $\ve x^*_B = A_B^{-1}\ve b$ and
$x^*_i = 0$ for all $i \in N:=\{1,\ldots,n\} \setminus B$. In particular, $\supp(\ve x^*)$ is bounded by $m$ alone. As we will see in this paper (cf. \cite{ADOO}), the story is much
more complicated for integer linear programs.

%Before stating our main results, we will introduce the following notation. 

Throughout this paper $\log(\cdot)$ will denote the logarithm with base two and $\|A\|_\infty$  will denote the maximum absolute entry of the matrix $A$.
The first result of this paper shows that Problem \ref{eq:mainProblem} has an optimal solution with support of size not exceeding a bound that only 
depends on $m$ and $\|A\|_\infty $ and independent of $n$. % the number of columns of $A$.
\begin{theorem} [Upper bound on the discrete support function]
\label{thm:UpperBound}
There exists  an optimal solution $\ve z^*$ for Problem~\eqref{eq:mainProblem} such that
\begin{equation}\label{t:Bound_via_determinant_bound1}
|\supp(\ve z^*)|  \le m + \log\left(g^{-1}\sqrt{\det(AA^T)}\right) \le  2m\log(2\sqrt{m}\|A\|_\infty),
\end{equation}
where $g$ denotes the greatest common divisor of all $m \times m$ minors of $A$.
\end{theorem}
This theorem  cannot be directly derived from the statements in \cite[Theorem 1]{EisenbrandShmonin2006} and \cite[Theorem~1.1]{ADOO} for the following reason.
To guarantee optimality of ${\ve y^*}$, the objective function vector $\ve c$ would have to become part of the constraint matrix (see \cite[Corollary~1.3]{ADOO}).  
This, however, does not allow us to bound the support solely in terms of $A$. The resulting bound would have to depend also on $\ve c$.

Our strategy to prove Theorem \ref{thm:UpperBound} is based on a refinement of the proofs for results established in  \cite[Theorem 1]{EisenbrandShmonin2006} and \cite[Theorem~1.1]{ADOO}.
Specifically, in  \cite[Theorem 1]{EisenbrandShmonin2006} it was shown that there exists an integer point $\ve y^*\in P(A, {\ve b})$ 
with support bounded by $ 2m \log(4m \|A\|_\infty).$ This bound was recently improved in \cite[Theorem~1.1]{ADOO} to
\begin{equation}\label{eq:feasibilityBound}
|\supp(\ve y^*)| \le m + \log\left(g^{-1}\sqrt{\det(AA^T)}\right)\,. %\leq  2m\log(2\sqrt{m}\|A\|_\infty).
\end{equation}
%
%Next we want to discuss how good the upper bound is.
Notice that an upper bound on the support of an optimal solution in the form of a function in $m$ cannot be expected. 
This would imply that, for fixed  $m$, one can optimize in polynomial  time by first guessing the support of an optimal 
solution and then finding a solution using Lenstra's algorithm for integer programming in fixed dimension \cite{Lenstra83}. However, also for fixed $m$, the optimization problem~\eqref{eq:mainProblem} is $NP$-hard \cite{MR519066}. 

The next result of this paper implies that the upper bound established in Theorem \ref{thm:UpperBound} is nearly optimal.
In fact, we consider a more general setting, where we drop the nonnegativity constraints and estimate from below the size of support 
of all integer solutions to a suitable underdetermined system $A{\ve x}={\ve b}$.
The obtained bound depends only on $m$ and $\|A\|_\infty$ and is independent of $n$ and $\ve b$. 
Under the nonnegativity constraints and provided that $m\ge 2$, a lower bound can be derived from an example given in \cite{EisenbrandShmonin2006}. 
%, and $\ve c$.
%This result also implies that the upper bound established in Theorem \ref{thm:UpperBound} is nearly optimal.
%
\begin{theorem}[Lower bound on the discrete support function]\label{thm:LowerBound}
For any $\epsilon>0$ there exists a natural number $N_\epsilon$ such that the following holds.
Let $m,n\in\N$, where $\frac{n}{m}>N_\epsilon$.
Then, there exists a matrix $A\in\Z^{m \times n}$, and a vector $\ve b\in\Z^m$, such that Problem~\eqref{eq:mainProblem} is feasible and
$$|\supp(\ve z)| \ge m\log(\|A\|_\infty)^{\frac{1}{1+\epsilon}},$$
for all ${\ve z}\in \Z^{n}$ with $A{\ve z}={\ve b}$.
\end{theorem}

Our main results have three implications that we outline next:

First we consider  the \emph{integer hull} $P_I$ of the polyhedron $P=P(A, {\ve b})$, i.e., $P_I=\conv(P\cap\Z^n)$.
It was proved by Cook et al. \cite[Theorem 2.1]{CookHartmannKannanDiarmid} that for a generic rational polyhedron $P$ the number of vertices of $P_I$
is bounded by $2m^n(n^2\log(\|A\|_\infty))^{n-1}$.
For polyhedra in standard form Theorem~\ref{thm:UpperBound} leads to an upper bound on the number
of vertices of the integer hull of $P$ that is polynomial in $n$, provided that $m$ is a constant.

To see this, let $\ve z^*$ be a vertex of $P_I$.
There exists  $\ve c\in\Z^n$, such that $\ve z^*$ is the unique optimum with respect to Problem~\eqref{eq:mainProblem}.
Taking a vertex $\ve x^*$ of $P$ that gives an optimal solution of the corresponding continuous relaxation,
we can associate every vertex of $P_I$ to at least one vertex of $P$.
Note, that $P$ itself has no more than $n^m$ vertices.
By \cite[Theorem 6]{EisenbrandWeismantel}, we can assume that $\|\ve x^*-\ve z^*\|_1\le m(2m\|A\|_\infty+1)^m$, where
$\|\cdot\|_1$ denotes the $\ell_1$-norm.

From Theorem~\ref{thm:UpperBound} it follows that every vertex $\ve z^*$ of $P_I$ has support of size at most $2m\log(2\sqrt{m}\|A\|_\infty)$.
Thus, $P_I$ has no more  than
\begin{equation*} 
n^m \cdot n^{2m\log(2\sqrt{m}\|A\|_\infty)} \cdot (m(2m\|A\|_\infty+1)^m)^{2m\log(2\sqrt{m}\|A\|_\infty)}
\end{equation*}
vertices.
Summarizing, we obtain

\begin{corollary}
The number of vertices of the integer hull $P_I$ is bounded by
$$(mn\|A\|_\infty)^{O(m^2\log(\sqrt{m}\|A\|_\infty))}.$$
\end{corollary}

A second consequence of Theorem \ref{thm:UpperBound} is the following extension to mixed integer optimization problems.

\begin{corollary}
There exists  an optimal solution $\ve z^*$ for the mixed-integer optimization problem
%\begin*{equation}
%\label{eq:mipProblem}
\[
\max\left\{ \ve c^T\ve x : A\ve x = \ve b, \, \ve x\ge\ve 0, \,\ve x \in \Z^n \times \R^d\right\},
\]
%\end*{equation}
		such that
	$$
	|\supp(\ve z^*)| \le 2 m + \log\left(g^{-1}\sqrt{\det(AA^T)}\right) \le  m+2m\log(2\sqrt{m}\|A\|_\infty).
$$
\end{corollary}

Given a positive integer $k$, an integer matrix $A$ is called $k$-{\em modular} if all its absolute subdeterminants are bounded by $k$. Theorem \ref{thm:UpperBound} immediately implies the following result.

\begin{corollary}
Fix a positive integer $k$. For a $k$-modular matrix $A$, Problem~\eqref{eq:mainProblem} has an optimal solution with the size of support bounded by a polynomial in $m$.
\end{corollary}

\section*{Proofs of Theorem~\ref{thm:UpperBound} and Theorem~\ref{thm:LowerBound}}

We will describe the refinement of the proofs established in  \cite{EisenbrandShmonin2006} and \cite{ADOO}.
In the proofs of both of these results the key idea is to show that if the support of a feasible solution is large, specifically if
the bound \eqref{eq:feasibilityBound} is violated, then there exists a certain nonzero integer vector $\ve y\in\{-1,0,1\}^n$
in  $\ker(A):=\{{\ve x}\in \R^n: A\ve x= \ve 0\}$. This vector can then be used to perturb the feasible solution and decrease its support.
We show that if  \eqref{t:Bound_via_determinant_bound1} is violated, then there exists such a vector $\ve y$ that
satisfies additionally $\ve c^T \ve y=0$. This allows us to perturb an optimal solution and decrease its support
while remaining optimal. As in \cite{ADOO}, we will utilize the following result by Bombieri and Vaaler.

\begin{theorem}[{\cite[Theorem~2]{BombVaal}}]\label{eq:sl_sl_f}
There exist $n - m$ linearly independent integral vectors ${\ve y}_1, \ldots,{\ve y}_{n-m} \in \ker(A)\cap\Z^n$ satisfying
\begin{equation*}
 \prod_{i=1}^{n-m} \| {\ve y}_i\|_{\infty} \le g^{-1}\sqrt{{\rm det}(AA^{T})}\,,
 \end{equation*}
where $g$ is the greatest common divisor of all $m \times m$ subdeterminants of $A$.
\end{theorem}

\begin{proof}[Proof of Theorem~\ref{thm:UpperBound}]
In \cite[Theorem 1.1 (ii)]{ADOO} it was verified that the second inequality in \eqref{t:Bound_via_determinant_bound1} always holds.
Therefore we only need to prove that the first inequality in \eqref{t:Bound_via_determinant_bound1} holds.
For that, let $\ve z^*$ be an optimal solution for Problem~\eqref{eq:mainProblem} with minimal support.

First, we argue that it suffices to consider the case $|\supp(\ve z^*)|=n$.
Suppose that $|\supp(\ve z^*)|<n$. For $S=\supp(\ve z^*)$ define
$\bar A=A_S$,  the submatrix of $A$ whose columns are labeled by the indices of $S$.
Let $\bar{\ve b}=\ve b$, $\bar{\ve c}=\ve c_S$, and $\bar{\ve z}^*=\ve z^*_S$.
By removing linearly dependent rows from $\bar A \bar{\ve x}=\bar{\ve b}$, we may assume that $\bar A$ has full row rank. Let $\bar m = \rank(\bar A) \le m$.
%We will assume that $\rank(\bar A)=m$, otherwise we remove redundant constraints from $\bar A \bar{\ve x}=\bar{\ve b}$.
%Thus $\bar A$ has a linearly independent rows.
Observe that  $\bar{\ve z}^*$ is an optimal solution for the corresponding Problem~\eqref{eq:mainProblem} with minimal support.
Furthermore, note that $\bar{\ve z}^*$ has full support.
Now, if \eqref{t:Bound_via_determinant_bound1} holds true for $\bar{\ve z}^*$, then by \cite[Lemma~2.3]{ADOO} we have
$$|\supp(\ve z^*)|=|\supp(\bar{\ve z}^*)|\le \bar m + \bar g^{-1}\sqrt{\det(\bar A\bar A^T)}\le m + g^{-1}\sqrt{\det(AA^T)},$$
where $\bar g$ denotes the greatest common divisor of all $\bar m \times \bar m$ minors of $\bar A$.

From now on, suppose that $|\supp(\ve z^*)|=n > m + \log(g^{-1}\sqrt{\det(AA^T)})$. This inequality implies that $g^{-1}\sqrt{\det(AA^T)}<2^{n-m}$.
By Theorem \ref{eq:sl_sl_f}, there exists a vector $\ve y\in\Z^n\setminus\{\ve 0\}$ such that
$$A\ve y=\ve 0 \;\; \text{ and } \;\; \|\ve y\|_\infty\le\left(g^{-1}\sqrt{\det(AA^T)}\right)^{1/(n-m)}<2.$$
In particular, this implies that $y_i\in\{-1,0,1\}$ for all $i\in\{1,\ldots,n\}$.
Since $z^*_i > 0$ for all $i \in \{1,\ldots,n\}$, it follows that $\ve z^* + \ve y$ and $\ve z^* - \ve y$ are feasible.
As $\ve z^*$ is optimal, it also follows that $\ve c^T\ve y = 0$. Otherwise $\ve c^T(\ve z^* + \ve y)$ or $\ve c^T(\ve z^* - \ve y)$ would be larger.

We may assume that $y_i<0$ for at least one $i$. Otherwise, replace $\ve y$ by $-\ve y$.
Let $\lambda=\min\{z^*_i : y_i<0\}$.
Now, by construction $\ve z^* + \lambda \ve y$ is an optimal solution to \eqref{eq:mainProblem} with support strictly smaller than $n$.
The obtained contradiction completes the proof.
\end{proof}

Now we discuss the proof of Theorem~\ref{thm:LowerBound}.
Assuming that the columns of $A$ form a Hilbert basis, Cook et al. \cite{MR830593} showed that the support of an optimal solution can be bounded solely in terms of $m$.
In the same article the authors constructed an example showing that such a result cannot hold for more general $A$. We extend their construction as follows.

Let $p_i\in\N$ denote the $i$-th prime number. The $i$-th \emph{primorial} $p_i\#$ is defined as the product of the first $i$ prime numbers, that is
$$p_i\#:=\prod_{j=1}^i p_j.$$
Based on the Prime Number Theorem (see e.g., \cite[Theorem 6]{hardy2008introduction}) the asymptotical growth of primorials is well known to be
\begin{equation}\label{asymptoticalGrowthOfPrimorials}
p_i\#=e^{(1+o(1))i\ln(i)},
\end{equation}
where $\ln(\cdot)$ denotes the natural logarithm (see \cite[Sequence A002110]{Sloane}).

\begin{proof}[Proof of Theorem~\ref{thm:LowerBound}]
%We may assume first that $k=\frac{n}{m}$ is integral. Otherwise we 
Let $k=\lfloor n/m \rfloor$ and let $p_1,\ldots,p_k\in\N$ be the first $k$ prime numbers.
We define
$$q_i:=\frac{p_k\#}{p_i}.$$
By construction, $\gcd(q_1,\ldots,q_k)=1$.
Thus, there exists a $\ve \lambda\in\Z^k$ such that $\sum_{j=1}^k \lambda_j q_j=1$.
Since, $\gcd(q_1,\ldots,q_{j-1},q_{j+1},\ldots,q_k)=p_j$ for any $j\in\{1,\ldots,k\}$, it
must hold that $\lambda_j\neq0$ for all $j\in\{1,\ldots,k\}$.
We set $\bar q_j:=-q_j$ if $\lambda_j<0$ and $\bar q_j:=q_j$ if $\lambda_j>0$.

Now, let $A\in\Z^{m \times n}$ be a matrix with %the following block structure.
%For each $i\in\{1,\ldots,m\}$ we define an independent constraint $\sum_{j=1}^k \bar q_j x_{j+ik-k}=1$.
%We set
$$A_{ij}=\begin{cases} \bar q_{j+k-ik}& \text{if } k(i-1)< j \le ik,\\
0&\text{otherwise}.\end{cases}$$

Further, let ${\ve 1 }$ denote the all-one $m$-dimensional vector. %we define $b\in\Z^m$ to be the all-one vector.
The vector $\ve z^*\in\Z^n$ with entries $z_i^*=|\lambda_{i-k\lfloor (i-1)/k \rfloor}|$ for $i\in\{1,\ldots,km\}$ and $z_i^*=0$ for $i\in\{km+1,\ldots,n\}$ is strictly positive in the first $km$ entries and satisfies $A\ve z^* = {\ve 1 }$.
So, Problem~\eqref{eq:mainProblem} with ${\ve b}={\ve 1}$ is feasible.
However, no integer vector $\ve z\in\Z^n$ exists that satisfies $A\ve z = {\ve 1 }$ and $|\supp(\ve z)|<km$.

It remains to note that by \eqref{asymptoticalGrowthOfPrimorials} $\|A\|_\infty=q_1=\frac{p_k\#}{2}=\frac{1}{2}e^{(1+o(1))k\ln(k)}$.
Thus, $\log(\|A\|_\infty)=\frac{\log(e)}{2}(1+o(1))k\ln(k)$.
Provided we have chosen $N_\epsilon$ sufficiently large, we have $\frac{\log(e)}{2}(1+o(1))k\ln(k)\le k^{1+\epsilon}$,
and therefore $|\supp(\ve z)| = m  k \ge m \log(\|A\|_\infty)^{1/(1+\epsilon)}$.
\end{proof}

%To conclude we note that  the paper \cite{EisenbrandShmonin2006} contains a nice example that leads to a strong lower bound, provided that $m\ge2$.
\vskip .3cm

\noindent {\bf Acknowledgements:}  This work was partially supported by NSF grant DMS-1440140, 
while the second author was in residence at the Mathematical Sciences Research Institute in Berkeley, 
California, during the Fall 2017. The second author was also partially supported by NSF grant DMS-1522158.

The third author is supported by the Swiss National Science Foundation
(SNSF) within the project \emph{Convexity, geometry of numbers, and
  the complexity of integer programming (Nr. 163071)}. This work was done in part while the third author was visiting the 
  Simons Institute for the Theory of Computing, partially supported by the DIMACS/Simons Collaboration on 
  Bridging Continuous and Discrete Optimization through NSF grant \#CCF-1740425.

The work of the fifth author was supported by an Alexander von Humboldt research award.

\small
\bibliographystyle{plain}
\bibliography{bib-3}

\end{document}